\documentclass[12pt, a4paper]{amsproc}

\newcommand{\normal}{\triangleleft}
\newcommand{\W}{{\mathfrak W}}

\newcommand{\V}{{\mathfrak V}}

\newcommand{\Ni}{{\mathfrak N}}
\newcommand{\B}{{\mathfrak B}}

\newcommand{\X}{{\mathfrak X}}
\newcommand{\Y}{{\mathfrak Y}}
\renewcommand{\L}{{\mathfrak L}}

\newcommand{\U}{{\mathfrak U}}
\newcommand{\A}{{\mathfrak A}}

\newcommand{\Z}{{\mathbb Z}}
\newcommand{\F}{{\mathbb F}}

\newcommand{\1}{\{1\}}

\newcommand{\var}[1]{\mathrm{var}\left( #1 \right)}
\newcommand{\varr}[1]{\mathrm{var}( #1 )}

\newcommand{\Wrr}{\,\mathrm{wr}\,}

\theoremstyle{definition}
  
\theoremstyle{plain}
\newtheorem{Lemma}{\sc Lemma}  
\newtheorem{Theorem}{\sc Theorem}  
  
\newtheorem{Corollary}{\sc Corollary}  
  
\theoremstyle{remark}
 
\newtheorem{Example}{\sc Example} 

\begin{document}

$\phantom{Line}$ 


\title[The Criterion of Shmel'kin and Varieties]{\sc \large The Criterion of Shmel'kin and Varieties Generated by Wreath Products of Finite Groups}

\author{Vahagn H.~Mikaelian}

\thanks{Partial results of this research were presented to the  {\it ``Mal'tsev Meeting''} International Conference, Novosibirsk, Russia, November 10--13, 2014. 
\,
The author was supported in part by joint grant 15RF-054 of RFBR and SCS MES RA (in frames of joint research projects SCS  and RFBR), and by 13-1A246 grant of SCS MES RA}



\email{v.mikaelian@gmail.com}

\date{\today}


\begin{abstract}
We present a general criterion under which the equality $\var{A \Wrr B} = \var{A} \var{B}$ holds for finite groups $A$ and $B$. This generalizes known results in this direction in the literature, and continues our previous research on varieties generated by wreath products of abelian groups. The classification is based on technics developed by A.L.~Shmel'kin, R.~Burns et al.~to study the critical groups in nilpotent-by-abelian varieties. 
\end{abstract}

\maketitle


\section{Introduction}
\label{Introduction}

\noindent
Our aim is to present a criterion classifying all the cases, when for the finite groups $A$ and $B$ their standard wreath product $A \Wrr B$ generates the product $\var{A} \var{B}$ of varieties generated by $A$ and $B$ respectively. Under wreath products we mean standard direct wreath products, which in this case coincide with standard Cartesian wreath products, since the groups are finite (and the criterion holds for both Cartesian and direct wreath products). We prove:

\begin{Theorem}
\label{Theorem wr finite}
For finite non-trivial groups $A$ and $B$ the equality 
\begin{equation}
\label{EQUATION_main}    
\var{A \Wrr B} = \var{A} \var{B}
\end{equation}
holds if and only if: 
\begin{enumerate}
  \item[a)]  \vskip-1mm the exponents of group $A$ and $B$ are coprime;
  \item[b)] $A$ is a nilpotent group, $B$ is an abelian group;  
  \item[c)] $B$ contains a subgroup isomorphic to the direct product $C_n^c$ of $c$ copies of cycle $C_n$, where $c$ is the nilpotency class of $A$, and $n$ is the exponent of $B$.
\end{enumerate}
\end{Theorem}

In theory of varieties of groups the study of equality \eqref{EQUATION_main} for (not necessarily finite) groups is motivated by the importance of wreath products as tools to study the product varieties of groups. 
For varieties $\U$ and $\V$ their product $\U \V$ consists of all possible extensions of all groups $A \in \U$ by all groups $B \in \V$. 
Take $A$ and $B$ to be some fixed groups generating the varieties $\U$ and $\V$ respectively. Then, if 
$\var{A \Wrr B}=\U \V$ holds, 
 we can restrict ourselves to consideration of $\var{A \Wrr B}$, which is easier to study rather than to explore all the extensions in $\U \V$.
In literature there are very many applications of the above approach: one may check Hanna Neumann's monograph~\cite{HannaNeumann} for examples and for references to other articles.

One of the first results in that direction was proved by G.~Higman (Lemma 4.5 and Example 4.9 in~\cite{Some_remarks_on_varieties}): the equality 
$\var{C_p \Wrr C_n}= \var{C_p} \var{C_n}= \A_p \A_n$ 
holds for any finite cycles $C_p$ and $C_n$ provided that $p$ is a prime not dividing $n$  (as usual $\A_n$ denotes the variety of all abelian groups of exponent dividing $n$). 
C.H.~Houghton generalized this for the case of arbitrary finite cycles $A=C_m$ and $B=C_n$. Namely, 
$\var{C_m \Wrr C_n}= \var{C_m}  \var{C_n} = \A_m \A_n$ 
holds if and only if $m$ and $n$ are coprime (this is mentioned in~\cite{Burns65}, \cite{HannaNeumann}, etc.).

In articles~\cite{AwrB_paper}--\cite{wreath products algebra i logika} we presented full classification of all cases when  \eqref{EQUATION_main} holds for arbitrary abelian groups. In~\cite{Metabelien} we gave a classification of all cases when the analog of \eqref{EQUATION_main} holds for wreath products of sets of abelian groups.
After the classification was found for all abelian groups, it is natural to widen the class of groups, and the first class to consider are {\it finite groups}. In the listed papers we already had suggested some special cases such as examples 8.5, 8.6 and 8.7 in~\cite{Metabelien}, Proposition 2 and Example 2 in~\cite{wreath products algebra i logika} in which   the analog of \eqref{EQUATION_main} holds or does not hold for some specific non-abelian finite groups.

An intriguing fact additionally motivating this study is the well known theorem of A.L. Shmel'\-kin who proved that the product $\U \V$ of non-trivial varieties $\U$ and $\V$ can be generated by a finite group if and only if 
the exponents of $\U$ and $\V$ are non-zero and coprime, $\U$ is a nilpotent variety, and $\V$ is an abelian variety~\cite[Theorem 6.3]{ShmelkinOnCrossVarieties}. Since for finite groups $A$ and $B$ the wreath product $A \Wrr B$ also is finite, we already have necessity of the conditions (a) and (b) in Theorem~\ref{EQUATION_main}. That is, we just have to distinguish those pairs of nilpotent groups $A$ and of abelian groups $B$, which satisfy the criterion of Shmel'kin, and for which $A \Wrr B$ generates $\var{A} \var{B}$. We have intentionally formulated Theorem~\ref{EQUATION_main} so that it is as alike to Theorem 6.3 in~\cite{ShmelkinOnCrossVarieties} as possible.

Another related result is the important theorem of R.~Burns on the base rank of the variety $\Ni_{c,m}\A_n$, where  $\Ni_{c,m} = \Ni_c \cap \B_m$ is the variety of all nilpotent groups of class at most $c$ and of exponents dividing $m$. Recall that the base rank $l(\V)$ of a variety $\V$ is defined to be the minimal (finite or countable) rank $l$ for which $F_l(\V)$ generates $\V$. R.~Burns has proved that $l=l(\Ni_{c,m}\A_n)=c$, whenever $m$ and $n$ are coprime~\cite{Burns65}.
Using technics with critical groups in~\cite[Section 3]{Burns65} one could easily find cases when the equality \eqref{EQUATION_main} holds or does not hold, say, for $A=F_2(\Ni_{2,p})$ and $B = C_q^k$, where prime numbers $p$ and $q$ are chosen so that $q$ divides $p-1$. Applying methods from our previous research \cite{AwrB_paper}--\cite{wreath products algebra i logika} we generalize this in Theorem~\ref{Theorem wr finite} for {\it arbitrary} finite $A$ and $B$.

Theorem~\ref{Theorem wr finite} has especially simple shape, when $A$ is abelian, that is, when $c = 1$. Then the condition (c) of Theorem~\ref{Theorem wr finite} means that $B$  contains an element of order $n = \exp B$. Since this holds for any finite abelain group, the only point we actually have to check in  Theorem~\ref{Theorem wr finite} is condition (a):

\begin{Corollary}[Theorem 4.6 in~\cite{Metabelien}]
\label{Corollary finite}
For arbitrary finite abelian groups $A$ and $B$ of exponents $m$ and $n$ respectively the equality
$\var{A \Wrr B} =  \A_m \A_n$ holds if and only if $m$ and $n$ are coprime.
\end{Corollary}

As we had mentioned in~\cite{Metabelien} the statement above seems to be a fact known in mathematical folklore. When $A$ and $B$ are cyclic, we get Theorem of Houghton (from which the previous corollary also may be deduced):

\begin{Corollary}[Theorem of Houghton]
\label{Corollary Houghton}
For arbitrary cyclic groups $C_m$ and $C_n$ of orders $m$ and $n$ respectively the equality
$\var{C_m \Wrr C_n} =  \A_m \A_n$ holds if and only if $m$ and $n$ are coprime.
\end{Corollary}

\section{Sufficiency of the condition of Theorem~\ref{Theorem wr finite}}  
\label{Sufficiency of the condition}

Following the conventional notation in theory of varieties of groups for a given class $\X$ of groups we respectively denote by  ${\sf Q}\X$, ${\sf S}\X$ and  ${\sf C}\X$ the classes of all homomorphic images, subgroups and cartesian products of groups of $\X$. By Birkhoff's Theorem~\cite{BirkhoffQSC,HannaNeumann} for any class $\X$ of groups the variety $\varr{\X}$ generated by it can be obtained by these three operations: $\varr{\X}={\sf QSC}\,\X$.

For the given classes of groups $\X$ and $\Y$ denote $\X \Wrr \Y = \{ X\Wrr Y \,|\, X\in \X, Y\in \Y\}$. 
We need two lemmas combining a few statements, which either restate some known facts in the literature or are proved by us earlier 
(see Proposition 22.11 and Proposition 22.13 in \cite{HannaNeumann},
Lemma 1.1 and Lemma 1.2 in \cite{AwrB_paper}
and also \cite{ShmelkinOnCrossVarieties} and \cite{BrumbergOnWreathProducts}). We omit the proofs, which can be found in  \cite{AwrB_paper}.

\begin{Lemma}
\label{X*WrY_belongs_var}
For arbitrary classs $\X$ and $\Y$ of groups and for arbitrary groups $X^*$ and $Y$, where either $X^*\in 
{\sf Q}\X$, or $X^*\in {\sf S}\X$, or $X^*\in {\sf C}\X$, and where $Y\in \Y$, the group $X^* \Wrr
Y$  belongs to the variety $\var{\X \Wrr \Y}$.
\end{Lemma}

\begin{Lemma}
\label{XWrY*_belongs_var}
For arbitrary classs $\X$ and $\Y$ of groups and for arbitrary groups $X$ and $Y^*$, where $X\in\X$ and 
where $Y^*\in {\sf S}\Y$, the group $X \Wrr Y^*$  belongs to the variety $\var{\X \Wrr \Y}$.
Moreover, if  $\X$ is a class of abelian groups, then for each  $Y^*\in {\sf Q}\Y$ the group $X
\Wrr Y^*$  also belongs to $\var{\X \Wrr \Y}$.
\end{Lemma}

Since the Cartesian and direct wreath products of any groups generate the same variety of groups~\cite{HannaNeumann}, the analogs of both lemmas also hold for direct wreath products.

Recall that a group is said to be {\it critical} if it is finite, and it is not in the variety generated by all its proper factors~\cite{HannaNeumann}. We will need the following lemma, which is based on ideas from~\cite{Burns65, HannaNeumann}:

\begin{Lemma}
\label{Lemma structure critical}
Let $A$ be a finite group of exponent $m$ and of nilpotency class $c$, and let $B$ be any finite abelian group of exponent $n$ coprime to $m$. Then any non-abelian critical group in the product variety $\W = \var{A}\var{B}$ is an extension of a group from $\var{A}$ by means of an at most $c$-generator group from $\var{B}$.
\end{Lemma}

\begin{proof}
Let $K$ be any of non-abelian critical groups in $\W$. Denote by ${\rm F} = Fit(K)$ the Fitting subgroup of $K$, that is, the unique maximal nilpotent normal subgroup of $K$. Since ${\rm F}$ is finite, it is a direct product of its Sylow subgroups, which are characteristic in ${\rm F}$ and thus normal in $K$. 
If  ${\rm F}$ had more than one such Sylow subgroups, they would intersect trivially, and
 $K$ would be embeddable into the direct product of its factor groups by these Sylow subgroups. Since $K$ is critical, it has one Sylow subgroup only: ${\rm F}$ is a $p$-group. 

The centralizer of a Fitting subgroup in each soluble group is contained in the Fitting subgroup~\cite[Theorem 1.3, Chapter 6]{Gorenstein}.

Denote by $\Phi = Frat(K)$ the Frattini subgroup of $K$: the set of non-generators of $K$, or the intersection of all maximal subgroups of $K$. The Frattini subgroup of a finite group is nilpotent and, since it also is normal (in fact also characteristic), $\Phi$ is a subgroup of ${\rm F}$. By ~\cite[Lemma 52.42]{HannaNeumann} (see also~\cite[Theorem 5.2.15 (ii)]{Robinson}) the Fitting subgroup $Fit(K/{\Phi})$ is equal to ${\rm F} / \Phi$. Since $K$ is finite, the factor ${\rm F} / \Phi$ is a direct products of some finitely many copies of a finite cycle $C_p$ by the result we just cited. Denoting their generators by $z_1, \ldots, z_l$ we get a presentation of ${\rm F} / \Phi$ as a vector space over the field $\F_p$:
$$
{\rm F} / \Phi \cong \Phi z_1 \oplus \cdots \oplus \Phi z_l.
$$
Assume $K$ is an extension of a group $L \in \var{A}$ by the group $T \in \var{B}$. 
Since $L$ is nilpotent and normal, $L \le {\rm F}$ holds.
On the other hand, since ${\rm F}/L$ is a subgroup of $K/L \in \var{B}$, the exponent of ${\rm F}/L$ has to divide $n = \exp{B}$. Since ${\rm F}$ is a $p$-group with $p$ coprime to $n$, we have $F = L$.
Since $m$ and $n$ are coprime, by Schur-Zessenhaus Theorem $K$ contains a compliment of $F$  isomorphic to $T$. To keep notations simple, denote that compliment by $T$.

Actions of elements of $T$ on ${\rm F} / \Phi$ by conjugations define a linear representation of degree $l$ on the space ${\rm F} / \Phi $. By Maschke's Theorem the latter is a sum of some irreducible subspaces $\Phi Z_i$:
$$
{\rm F} / \Phi \cong \Phi Z_1  \oplus \cdots \oplus  \Phi Z_s, \quad s \le l,
$$
which defines $s$ irreducible linear representations of $T$, if we restrict actions of $T$ upon $\Phi Z_i$. Denoting by $D_i$, $i=1,\ldots, s$, the kernels of these representations we get faithful representations for each of  $s$ factor-groups $T / D_i$.  Their intersection $D = \bigcap_{i=1}^s D_i$ is trivial because a non-trivial element from $D$ would centralize ${\rm F} / \Phi$, whereas ${\rm F} / \Phi = Fit(K/{\Phi})$, and an element outside the Fitting subgroup cannot centralize it in a soluble group, as we mentioned above. 

An abelian group with faithful representation is cyclic. This means that the group $T \cong \Phi  T / \Phi$ is embeddable into the direct product of at most $s$ finite cycles, and the number of generators of $T$ is restricted by $s$.

\vskip1mm
Now we can use Corollary~\cite[51.38]{HannaNeumann} of an important theorem of S.~Oates and M.B.~Powell~\cite{Oates Powell} (this theorem is mentioned in~\cite{HannaNeumann} as Theorem 51.37). 
Our critical group $K$ possesses a normal nilpotent subgroup ${\rm F}$ with a compliment $T$, and ${\rm F}$ has normal subgroups $\Phi Z_1, \ldots , \Phi Z_s$ such that 
{\parskip0mm

(i) each $\Phi Z_i$ is closed under conjugations of $T$;

(ii) $K = \langle \Phi Z_1, \ldots , \Phi Z_s, \, T \rangle$ holds;

(iii) no proper subset of the set $\{\Phi Z_1, \ldots , \Phi Z_s \}$ together with $T$ generates $K$.
}
Then by~\cite[51.38]{HannaNeumann} the number $s$ is less than or equal to the nilpotency class of ${\rm F}$. So the number of generators of $T$ also is not greater  than $c$.
\end{proof}

Now we can prove the sufficiency of the condition of Theorem~\ref{Theorem wr finite}:

\begin{proof}[Proof of Theorem~\ref{Theorem wr finite}, sufficiency]
Under conditions of the theorem both varieties $\var{A}$ and $\var{B}$ are locally finite, so by  theorem of O.Yu.~Schmidt~\cite{HannaNeumann} the product $\W = \var{A} \var{B}$ also is locally finite. By~\cite[Proposition 51.41]{HannaNeumann} $\W$ is generated by its critical groups.

Take $K\in \W$ to be any of such critical groups. If $K$ is abelian, then it is a cyclic $p$-group for some prime $p$~\cite[Propositon 51.36]{HannaNeumann}. If $p$ is a divisor of $m$, then $K= C_p \cong C_p \Wrr \1$, and if $p$ is a divisor of $n$, then $K= C_p \cong \1 \Wrr C_p$. In both cases $K$ belongs to $\W$ by Lemma~\ref{X*WrY_belongs_var} or by  Lemma~\ref{XWrY*_belongs_var} for $\X =\{ A \}$ and $\Y =\{ B \}$.

Let $K$ be non-abelian, and assume it is the extension of the group $L \in \var{A}$ by the group $T \in \var{B}$. By Lemma~\ref{Lemma structure critical} we may assume $L$ to be an at most $c$-generator group. Thus it is a subgroup of the direct power $C_n^c$ and by requirement of the theorem it is contained in $B$.  On the other hand, $L$ can be obtained from $A$ by means of operations ${\sf Q,S,C}$. Applying Lemma~\ref{X*WrY_belongs_var} and Lemma~\ref{XWrY*_belongs_var}  for the same $\X$ and $\Y$ we get that $K\in \var{A \Wrr B}$, which completes the proof.
\end{proof}

\section{Necessity of the condition of Theorem~\ref{Theorem wr finite}}  
\label{Necessity of the condition}

Below we without any definitions use the concepts of verbal products and verbal wreath products. The first can be found in the papers of S.~Moran~\cite{Moran 1}-\cite{Moran 3}, and for the second one can check the papers of A.L.~Shmel'kin~\cite{ShmelkinOnCrossVarieties} or R.~Burns~\cite{Burns65}. Both concepts also are presented in~\cite{HannaNeumann}.

Take a group $A$ of nilpotency class $c$ and of exponent $m$, and denote for briefness $\L = \var{A}$. It is clear that  $\L$ is a subvariety of the variety $\mathfrak N_{c,m}$ of all groups of nilpotency class at most $c$ and of exponent dividing $m$. Fix any prime $p$ not dividing $m$.
We need a specific group similar to the group $W_c$ used in~\cite{Burns65}, but in our case it is a slightly different group. Define the $\L$-verbal wreath product:
$$
W(A,p)=F_c(\L) \Wrr_{\!\! \L} \, C_p^c
$$
of the free group $F_c(\L)$ of rank $c$ in the variety $\L$ and of the direct product of $c$ copies of the cycle $C_p$, where $c$ is the nilpotency class of $A$. In the case, when $A$ generates the variety $\mathfrak N_{c,m}$, the group $W(A,p)$ is the group $W_c$ of~\cite{Burns65} for prime $n = p$. Most of the steps of the construction below are similar to the steps of~\cite{Burns65} or of~\cite{ShmelkinOnCrossVarieties}. 

The base subgroup of $W(A,p)$ is the $\L$-verbal product 
\begin{equation}
\label{the base}
\prod_{c \in C_p^c}^{\hskip5mm \L} F_c(\L) 
\cong \prod_{c \in C_p^c}^{\hskip5mm \L} \Big( \!\!\!\!\!\! \prod_{\hskip4mm i = 1,\ldots , c}^{\hskip5mm \L} \!\!\!\! C_m \Big)
\cong \!\!\! \!\!\! 
\prod_{\hskip2mm i = 1,\ldots , c  p^{c}}^{\hskip5mm \L} \!\!\! \!\!  C_m 
\end{equation}
(we used the associativity of verbal products). Present $F_c(\L)$ as the factor group $F_c/L(F_c)$ of the absolutely free group $F_c = F_c(x_1, \ldots, x_c)$ by the verbal subgroup $L(F_c)$.
If in the verbal product \eqref{the base} for each element $b \in C_p^c$ we denote by $x_i(b)$ the $b$'th copy of $x_i$, we can interpret \eqref{the base} as a factor group of the absolutely free group $F_{c  p^{c}}$ with $c  p^{c}$ generators $\{ x_i(b) \, | \, i=1,\ldots, c; \,\, b \in C_p^c \}$ by the verbal subgroup $L (F_{c  p^{c}})$.
For briefness of later use denote $L (F_{c  p^{c}})$ by $L'$.

In $F_c$ there is such a basic commutator $\gamma(x_1, \ldots, x_c)$ of weight $c$ that $\gamma\big(F_c(\L)\big)$ is non-trivial because otherwise $F_c(\L)$ would be a group not of class $c$ but of class $c-1$.
Consider the set   
$\Gamma = \{\gamma\big(x_1(b_1), \ldots, x_c(b_c)\big) \, | \, b_i \in C_p^c \}$ of $s = p^{c^2}$ elements  and denote them by 
$
\gamma_1, \ldots , \gamma_{s}
$.
Since the class of $F_{c  p^{c}}(\L) \cong F_{c  p^{c}}/ L'$ also is $c$, any two elements $\gamma_i$ and $\gamma_j$ (together with cyclic groups they generate) commute modulo $L'$.

To show that  the set $\Gamma$ modulo $L'$ generates in $F_{c  p^{c}}$ the direct product of cycles
\begin{equation}
\label{cycles}
\langle  L' \gamma_1 \rangle \times  \cdots \times \langle  L' \gamma_{s} \rangle 
\end{equation}
one need apply~\cite[Lemma 5.4.1]{Burns65} or just take any coset $ L' \gamma_{i_1}\cdots \gamma_{i_t}$, and for any of its factors $\gamma_{i_j}$ apply to $F_{c \cdot p^{c}}$ the homomorphism, which does not move the variables participating in $\gamma_{i_j}$, and sends all other variables to $1$. Thus each $\langle  L' \gamma_i \rangle$ intersects with the product of all other factors trivially.

Since the word $\gamma$ was applied on free generators, all the summands in \eqref{cycles} are cycles of the same non-trivial order $m'$. The latter divides $m$, and if we take any prime divisor $q$ of $m'$ and denote
$
v_i = \gamma_i^{m'/q}, 
$
we will get cycles $\langle L' v_i \rangle$  of order $q$ for all $i=1,\ldots, s$. So \eqref{cycles} contains an $s$-dimensional vector space over the field  $\F_q$:
\begin{equation}
\label{space}
V= \F_q^{s} \cong
\langle L' v_1 \rangle \times  \cdots \times \langle L' v_{s} \rangle 
\end{equation}

With these constructions we have the following analog of Lemma 5.4 in~\cite{Burns65}:

\begin{Lemma}
\label{Lemma intersection}
Fix any group $A$ of nilpotency class $c$ and of exponent $m$, and denote $\L = \var{A}$. Then for any prime $p$ not dividing $m$ every non-empty set of normal subgroups of the verbal wreath product 
$
W(A,p)=F_c(\L) \Wrr_{\!\! \L} \, C_p^c
$, 
such that none of those normal subgroups is wholly contained in the base subgroup, has non-trivial intersection. 
\end{Lemma}

\begin{proof}
A product $(L' v_1^{x_1}) \cdots (L' v_s^{x_s}) = L' v_1^{x_1} \cdots  v_s^{x_s}$ with values $x_i$ inside the multiplicative group $W(A,p)$ is nothing else but the linear combination 
$
x_1 \cdot L' v_1 + \cdots + x_s \cdot L' v_s
$     
of vectors $L' v_i$ with scalars $x_i \in \Z_q$ in the additively written space $V$. This linear interpretation allows to find the non-trivial element in the intersection mentioned in the lemma.

We need the description of normal closure of subgroups of the active group inside the verbal wreath product found by R.~Burns in~\cite[Corollary 4.2]{Burns65}. Applying it to the group
$W(A,p)$ we get that for any non-trivial subgroup $U$ of $C_p^c$ (with transversal $T$ in $C_p^c$) the normal closure $U^{W(A,p)}$ of $U$ in the whole group is the product $UM$, where 
\begin{equation}
\label{the closure}
\begin{array}{ll} 
 \,\,\,\,\,\,\, 
M   & \!\!\! =
\Big\{
\alpha_1^{u_1} \cdots \alpha_r^{u_r} \, | \,
r \ge 1; \,\,
u_i \in U; \,\,
\alpha_i \in \prod_{t \in T}^{\hskip0mm \L} \big(F_c(\L)\big)(t), \, i=1,\ldots, r ; \\ 
  & 
\hskip12mm
u_j \not= u_{j+1}, \,\,\, j=1,\ldots, r-1 ; \,\,\,\,\,\,\, \alpha_1  \cdots \alpha_r =1 \Big\}
\end{array}
\end{equation}
(unlike the previous notation above, $\alpha_i^{u_i}$ means not the power of $\alpha_i$ but the shifting action of $u_i$ on $\alpha_i$ in verbal wreath product).

Let $U$ be a cycle of order $p$ in $C_p^c$. 
It is easy to bring a product $L' v_1^{x_1} \cdots v_s^{x_s}$ to the shape mentioned in \eqref{the closure}. Namely, consider a word 
$
\gamma(x_1(b_1), \ldots, x_c(b_c))
$
participating in this product (for some selection of values $b_1, \ldots, b_c \in C_p^c$). 
For each $x_i(b_i)$ (or its power) participating in $\gamma$ find the representative $t_i$ of the coset $U b_i$ (such that $b_i =  t_i u_i^{-1}$ for some $u_i \in U$), and replace $x_i(b_i)$ by $x_i^{u_i}(t_i)$. So $\gamma$ will be presented as a product of elements $\alpha_i$ mentioned in \eqref{the closure}, each shifted by some elements $u_i \in U$. If a few neighbor elements use the same $u_i$, we can merge them to one $\alpha_i$ to have the condition $u_j \not= u_{j+1}$, as well.

After these transformations distinct elements $L' v_1^{x_1} \cdots v_s^{x_s}$, of course, may merge. If so, then  using the additive notation in $V$ collect the identified vectors $L' v_i^{x_i}$. By linear independence of vectors $L' v_i$ the condition $\alpha_1  \cdots \alpha_r =1$ of \eqref{the closure} will just mean that the integers $x_i$ form a solution $(x_1, \ldots, x_c)$ for a system of $|T|^c = p^{(c-1)c}$ homogeneous linear equations over $\F_q$.
Notice that during our manipulations we never change the $i$ index of $x_i(b_j)$ inside any $\gamma$. We just shift the $b_j$, so none of the non-zero words $\gamma$ will be mapped to zero.

$C_p^c$ contains $(p^c-1)/(p-1)$ cycles of order $p$. Each normal subgroup of $W(A,p)$, not wholly contained in the base subgroup, contains one of the cycles mentioned. And an element belongs to all of the normal subgroups if and only if the combined system of all the $(p^c-1)/(p-1) \times p^{(c-1)c}$ linear equations has a solution over $\F_q$. Recall that the equations are on $s = p^{c^2}$ variables which is larger then the number of equations.  So this system of homogeneous linear equations will always have a non-zero solution.
\end{proof}

The completed lemma allows to prove the remaining part of Theorem~\ref{Theorem wr finite}:

\begin{proof}[Proof of Theorem~\ref{Theorem wr finite}, necessity]
If the given finite groups $A$ and $B$ do not meet the condition of the theorem of Shmel'kin, then $\var{A} \var{B}$ cannot be generated by a finite group, so it is not equal to $\var{A \Wrr B}$ also. Thus we can restrict ourselves by the remaining case only: 
$A$ is nilpotent, $B$ is abelian, both groups are finite, the exponents  $m = \exp A$ and $n = \exp B$ are coprime, but the number of copies of the cycle $C_n$ in any direct decomposition of $B$ is less than $c$.

Take any prime divisor $p$ of $n$ and consider the group 
$
W(A,p)=F_c(\L) \Wrr_{\!\! \L} \, C_p^c = F_c(\L) \Wrr_{\!\! \L} \, F_c(\A_p).
$
Since 
$$
W(A,p) \in \L \A_p \subseteq \L \A_n = \var{A} \var{B},
$$ 
it will be enough, if we suppose that $\var{A} \var{B}$ is equal to $\var{A \Wrr B}$, and then arrive to a contradiction by showing that  $W(A,p)\notin \var{A \Wrr B}$ for a specific $p$.

According to a lemma of A.L.~Shmel'kin (see~\cite[Lemma 6.1]{ShmelkinOnCrossVarieties}), if the exponents of arbitrary locally finite varieties $\U$ and $\V$ are coprime, then the $\U\V$-free group $F_{2r}(\U\V)$ of rank $2r$ of the product variety $\U\V$ contains the $\U$-verbal wreath product 
$
F_j(\U) \Wrr_{\!\! \U} \, F_r(\V)
$
for any $j \ge r$ (see also~\cite[Lemma 5.2]{Burns65}). Taking $\U = \L$, $\V = \A_p$ and $j = r = c$ we get that the relatively free group $F_{2c}(\L\A_p)$ contains the group $W(A,p)$.

By our assumption $F_{2c}(\L\A_p) \in \var{A \Wrr B}$, and we can apply~\cite[Theorem 15.4]{HannaNeumann}. Namely: $F_{2c}(\L\A_p)$ is embeddable into the Cartesian (in our case also direct, as the number of factors is finite) product 
$$
A \Wrr B^{\big((A \Wrr B) ^{2c}\big)}
$$
of $k = |(A \Wrr B) ^{2c}| = \big(|A|^{|B|} \cdot|B|\big) ^{2c}$ copies of the group $A \Wrr B$.

By the cited lemma of A.L.~Shmel'kin the group $W(A,p)$ also is embeddable into the direct product of $k$ copies of  $A \Wrr B$. So there are some normal subgroups $W_i \, \normal \, W(A,p)$, $i=1,\ldots, k$, with trivial intersection such that each factor $W(A,p) / W_i$ is isomorphic to some subgroup of $A \Wrr B$.

If none of the subgroups $W_i$ were wholly contained in the base subgroup of $W(A,p)$, then their intersection would be non-trivial by Lemma~\ref{Lemma intersection}. Therefore for some $i^*$ the subgroup $W_{i^*}$ is inside the base subgroup of $W(A,p)$. Thus the factor-group $W(A,p) / W_{i^*}$ contains a copy of the subgroup $C_n^c$.

Among the prime divisors of $n$ there is a $p$, such that $B$ does not contain the direct power $C_p^c$, for otherwise $B$ would also contain the direct product $C_n^c$ of all such $C_p^c$ over all $p$ dividing $n$. For this fixed $p$ the factor-group $W(A,p) / W_{i^*}$ contains a copy of $C_p^c$. By the construction above that copy is contained in $A \Wrr B$ also.
On the other hand the $p$-primary component $B(p)$ of the abelian group $B$ is the Sylow $p$-subgroup of $A \Wrr B$, as it is easy to see by comparing the group orders.

Thus, if $A \Wrr B$ contained a copy of $C_p^c$, that copy would be inside a conjugated isomorphic copy $B(p)^b$ of $B(p)$, which is impossible by the choice of $p$.
\end{proof}

\section{Some examples and applications}  
\label{examples applications}

\begin{Example}
\label{smallExample}
Recall that we above denoted $\Ni_{c,m} = \Ni_c \cap \B_m$. In~\cite{wreath products algebra i logika}, using the properties of critical groups from~\cite{Burns65} and Proposition 2 from~\cite{wreath products algebra i logika}, we saw for the group $A=F_2(\Ni_{2,3})$, that
\begin{equation}
\label{N_23a}
\var{A \Wrr C_2} \not= \var{A} \var {C_2} =\Ni_{2,3}  \A_2,
\end{equation}
\begin{equation}
\label{N_23b}
\var{A \Wrr (C_2 \oplus C_2)} = \var{A} \var{C_2 \oplus C_2} =\Ni_{2,3}  \A_2.
\end{equation}
As we mentioned in Introduction, using the construction of critical groups in~\cite[Section 3]{Burns65} one could build analogs of \eqref{N_23a} and \eqref{N_23b} for any variety $\Ni_{2,p}$ and $C_q$, where prime numbers $p$ and $q$ are chosen so that $q$ divides $p-1$. However, much more general cases can be covered by Theorem~\ref{Theorem wr finite} of current paper: for any $A=F_r(\Ni_{c,m})$ the equality 
$\var{A \Wrr C_n^r} = \var{A} \A_n$
holds if and only if $r$ is greater than or equal to the nilpotency class of $A$. And if $r \ge c$, the class of $A$ is $c$ because every nilpotent variety of class $c$ can be generated by its free group of rank $c$~\cite[Corollary 35.12]{HannaNeumann}. So in that case we additionally have 
$\var{A \Wrr C_n^r} = \Ni_{c,m} \A_n.$
\end{Example}

\begin{Example}
An example of a finite nilpotent group of class $2$ is the dihedral group $D_4$ of order $8$. 
L.G.~Kov{\'a}cs in \cite{Kovacs dihedral} has computed the variety it generates: $\var {D_4}= \A_2^2 \cap \Ni_2$. So by Theorem~\ref{Theorem wr finite} for any odd $n$ we have $\var{D_4 \Wrr C_n} \subset (\A_2^2 \cap \Ni_2)\A_n$ and $\var{D_4 \Wrr (C_n \oplus C_n)} = (\A_2^2 \cap \Ni_2)\A_n$.
The quaternion group $Q_8$ of order eight generates the same variety as $D_4$ (see~\cite{HannaNeumann}), and it also is nilpotent of class $2$. So both the straight inclusion and the equality given above have analogs for $Q_8$.
\end{Example}

Our study of varieties generated by wreath products of groups started in our Ph.D. study, where  among other topics  a specific operation called $\circ$-product was considered. For the given variety $\V$ and the group $G$, the $\circ$-product $\V \circ G$ is defined as the variety generated by all possible extensions of all groups from $\V$ by the group $G$. Since for any group $N$ generating $\V$ and for any extension $E$ of $N$ by the group $G$ the $\circ$-product $\V \circ G$ contains $E$ and is contained in the product variety $\V \, \var{G}$,  the $\circ$-product  is a somewhat sharper tool (than the conventional product of varieties of groups) to study the extensions of groups. To use the potential of this operation we classified some cases, when $\V \circ G = \V \, \var{G}$ holds. 
In particular, by Theorem 5.6  in our Ph.D. thesis for any regular variety $\V$ and finite group $G$ the $\circ$-product $\V \circ G$ is equal to $\V \, \var{G}$ only if $G$ is abelian (a variety is called regular, if none of its relatively free groups can be embedded into a relatively free group of lower rank~\cite{HannaNeumann}). Theorem~\ref{Theorem wr finite} of the current work allows to get more details for a particular case:

\begin{Corollary}
Let $\V$ be any non-trivial variety generated by a finite group and let $G$ be any non-trivial finite group. Then the equality $\V \circ G = \V \, \var{G}$ holds if and only if 
the exponents of group $\V$ and $G$ are coprime,
$\V$ is a nilpotent variety, $G$ is an abelian group,  
and a direct decomposition of $G$ contains at least $c$ copies of the cycle $C_n$, where $c$ is the nilpotency class of $\V$, and $n$ is the exponent of $G$.
\end{Corollary}

When this manuscript was under preparation, we had an opportunity to discuss it with Prof.~A.Yu.~Ol'shanskii.  His very helpful comments allowed to improve details of the text.


\vskip3mm

\end{document}